\documentclass[12 pt]{amsart}
\usepackage[colorinlistoftodos]{todonotes}

\usepackage{amsxtra}
\usepackage{geometry}
\usepackage{amssymb}
\usepackage{cite}
\usepackage{graphicx}
\usepackage[mathscr]{eucal}
\RequirePackage{stix}

\usepackage{thmtools} 
\usepackage{thm-restate}

\usepackage{enumerate}
\newenvironment{enum}{  
\begin{enumerate}[\upshape(\arabic{section}.\arabic{equation}a)] }  
{  \end{enumerate}   }
\newcommand{\itemref}[2] {{\upshape(\ref{#1}\ref{#2})}}

\theoremstyle{plain}
\newtheorem{Thm}{Theorem}

\newtheorem*{MThm-rev}{Theorem B [expanded statement]}

\newtheorem{Cor}[Thm]{Corollary}
\newtheorem{Prop}[Thm]{Proposition}

\newtheorem{Q}[Thm]{Question}
\theoremstyle{definition}
\newtheorem{Remark}[Thm]{Remark}

\newtheorem{Ex}[Thm]{Example}

\newcommand{\belt}{\mathscr B}
\newcommand{\sss}{\smallskip}

\newcommand{\cut}{\setminus}

\renewcommand{\Re}{\operatorname{Re}}
\renewcommand{\Im}{\operatorname{Im}}
\renewcommand{\bar}{\overline}
\renewcommand{\tilde}{\widetilde}
\newcommand{\intl}{\int\limits}
\newcommand{\C}{\mathbb C}
\newcommand{\R}{\mathbb R}
\newcommand{\CP}{\mathbb{CP}}
\newcommand{\Z}{\mathbb Z}
\newcommand{\dee}{\partial}
\newcommand{\deebar}{\overline\partial}

\newcommand{\st}{\,:\,}

\newcommand{\lam}{\lambda}
\newcommand{\eps}{\varepsilon}
\newcommand{\mapdef}[4]{ #1 &\to #2 \\ #3 &\mapsto #4 }
\newcommand{\w}{\wedge}
\newcommand{\rest}{\big\rvert}
\numberwithin{equation}{section}

\begin{document}
\title[Projective-umbilic points]
{Projective-umbilic points of\\ circular real hypersurfaces in $\C^2$}
\author[David E. Barrett]{David E. Barrett}
\address{Department of Mathematics\\University of Michigan
\\Ann Arbor, MI  48109-1043  USA }
\email{barrett@umich.edu}

\author[Dusty E. Grundmeier]{Dusty E. Grundmeier}
\address{Department of Mathematics \\
Harvard University \\
 Cambridge, MA 02138-2901  USA }
\email{deg@math.harvard.edu}

\thanks{{\em 2010 Mathematics Subject Classification:} 	32V10}

\thanks{The first author was supported in part by NSF grant number DMS-1500142.}

\date{\today}

\begin{abstract}  
We show that the boundary of any  bounded strongly pseudoconvex complete circular domain in $\C^2$ must contain points that are exceptionally tangent to a projective image of the unit sphere.
\end{abstract}

\maketitle

\section{Background}\label{S:back}


A {\em vertex} of a smooth plane curve may be viewed as a point at which there
is a circle exceptionally tangent to to the curve; that is, there is a circle with fourth-order (or better) contact with the  curve at the vertex in contrast with with the situation of a non-vertex at which third-order contact is the best possible.  (What we are calling the order of contact here is also known as the "point-contact order" [Rut,\S 5.1.1]. An alternate convention -- used for example in the study of jets of functions -- reduces the orders by one.)

The famous four-vertex theorem ([Muk], [Kne], [Oss]) says the every smooth Jordan curve in the plane has at least four vertices.  There are corresponding results giving lower bounds for the size of the set of {\em affine vertices} where a curve is exceptionally tangent to a conic [Muk], for the set of {\em umbilic points} on a smooth non-toric compact surface in $\R^3$ where the surface is exceptionally tangent to a ball (see for example [Ber, pp.\,389-390], noting the unresolved Carath\'eodory conjecture), for the set of points where an orientation-preserving diffeomorphism of the unit circle is exceptionally tangent to a holomorphic automorphism of the unit disk ([Ghy], [OvTa]), and for the set of {\em CR-umbilic points} on the boundary of a real ellipsoid or bounded complete circular domain in $\C^2$ that are exceptionally tangent to a local biholomorphic image of the unit sphere ([HuJi], [EbSo]; note also the counter-example in [ESZ]).
(The order of contact at stake in the last batch of results is seven, and the points in question are those where a certain sixth-order tensor (due to Cartan [Car1], [Car2]) vanishes.)

In the current work we again consider real hypersurfaces in $\C^2$ or the projective space $\CP^2$, looking at orders of contact with projective images of the unit sphere (equivalently, with $\C$-affine images of the unit sphere or the Heisenberg hypersurface $\{\Im z_2=|z_1|^2\}$
).
The special points are the {\em projective-umbilic points} at which third-order contact is possible.

Due to the discrepancy in the collection of allowable maps on the one hand and the order of contact on the other hand, CR-umbilic points need not be projective-umbilic, nor vice versa
(see Example \ref{X:Lp} below.)

We show below that boundary of any bounded strongly pseudoconvex circular domain in $\C^2$ must contain projective-umbilic points.

\section{A Beltrami-style tensor}\label{S:Belt}

The following tensorial object will prove useful: for $S$ a smooth real strongly pseudoconvex hypersurface in $\C^2$ with defining function $r$ we set
\begin{equation}\label{E:belt}
\belt_S\eqdef -\frac{\det
{ \begin{pmatrix}
 0  & r_1  & r_2  \\
r_1  &  r_{1   1} & r_{2   1}  \\
r_{ 2 }  & r_{1   2}  &   r_{2   2} 
\end{pmatrix}
}}
{
\det \begin{pmatrix}
 0  & r_1  & r_2  \\
r_{\bar 1 }  &  r_{1 \bar 1} & r_{2 \bar 1}  \\
r_{\bar 2 }  & r_{1 \bar 2}  &   r_{2 \bar 2}  
\end{pmatrix}}\cdot
\frac{dz_1\w dz_2}{\overset{}{d\bar{z_1}\w d\bar{z_2}}}
\end{equation}
on $S$, where the subscripts denote differentiation.  (The non-vanishing of the denominator here is well-known to be equivalent to the strong pseudoconvexity of $S$
.)
The last factor above is to indicate that this object is to be
viewed as a section of the product of the canonical bundle of $(2,0)$-forms with
the conjugate-inverse of that bundle. We will refer to such objects as Beltrami differentials. (In a one-variable setting this reduces to the reciprocal of the Beltrami differentials $\mu(z)\,\frac{d\bar z}{dz}$ used in particular in the study of quasi-conformal mappings, as for example found in [Leh].  See also \S \ref{S:Rossi} below.)

\begin{Prop}\label{P:belt-basics}
\refstepcounter{equation}\label{N:belt-basics}
\begin{enum}
\item $\belt_S$ does not depend on the choice of defining function $r$. \label{I:def-fcn}
\item If $\psi$ is an automorphism of  $\CP^2$ then $\belt_S=\psi^*\left( \belt_{\psi(S)}\right)$ (where defined). \label{I:belt-inv}
\end{enum}
\end{Prop}

\begin{proof}
For \itemref{N:belt-basics}{I:def-fcn}, check that if $r$ is replaced by $\eta \cdot r$ with $\eta$ non-vanishing then both the numerator and denominator above pick up a factor of $\eta(p)^3$ at $p\in S$.

For \itemref{N:belt-basics}{I:belt-inv}, first recall that automorphisms of  $\CP^2$ have the form 
\begin{equation*}
(z_1,z_2)\mapsto
\left( \frac{D+Ez_1+Fz_2}{A+Bz_1+Cz_2}, \frac{G+Hz_1+Iz_2}{A+Bz_1+Cz_2} \right).
\end{equation*}
with $\begin{pmatrix}
A  & B  &  C \\
D  & E  &  F \\
 G &  H &   I
\end{pmatrix}$
invertible.   The group of these automorphisms is generated by the invertible affine transformations together with the particular transformation
\begin{equation*}
\left(z_1,z_2\right)\mapsto \left(\frac{1}{z_1},\frac{z_2}{z_1}\right);
\end{equation*}
the transformation law can be verified by straightforward computation in either case.
\end{proof}

In view of \itemref{N:belt-basics}{I:belt-inv}, the construction of $\belt_S$ also makes sense on a strongly pseudoconvex real hypersurface in $\CP^2$ -- see [Bar, \S5.3] for a more directly projective approach to the construction and the transformation law.

\begin{Prop}\label{P:sphere}
If $S_{\sf sph}$ is the unit sphere $\left\{(z_1,z_2)\in \C^2 \st |z_1|^2+|z_2|^2=1\right\}$ then $\belt_{S_{\sf sph}}=
0\cdot \frac{dz_1\w dz_2}{\overset{}{d\bar{z_1}\w d\bar{z_2}}}$.
\end{Prop}

\begin{proof}
This follows by direct computation with $r(z_1,z_2)=|z_1|^2+|z_2|^2-1$.
\end{proof}

\begin{Cor}\label{P:Heis}
If $S_{\sf Heis}$ is the Heisenberg hypersurface $\left\{(z_1,z_2)\in \C^2 \st \Im z_2=|z_1|^2 \right\}$ then $\belt_{S_{\sf Heis}}=
0\cdot \frac{dz_1\w dz_2}{\overset{}{d\bar{z_1}\w d\bar{z_2}}}$.
\end{Cor}

\begin{proof}
This can be handled either by direct computation as above or by applying \itemref{N:belt-basics}{I:belt-inv} to the projective automorphism
\begin{align*}
\mapdef{\psi\st S_{\sf Heis}}{S_{\sf sph}}{\left(z_1,z_2\right)}{\left(\frac{2z_1}{i+z_2},\frac{i-z_2}{i+z_2}\right).}
\end{align*}
\end{proof}

In the other direction we have the following result.
\begin{Thm}
\label{T:Bolt}\,  ([Jen], [DeTr], [Bol]).

$\belt_S$ vanishes identically if and only if $S$ is locally a projective image of the unit sphere $S_{\sf sph}$ (or equivalently, of the Heisenberg hypersurface $S_{\sf Heis}$).
\end{Thm}

\begin{Prop}\label{P:normal}
Let $S$ be a smooth strongly pseudoconvex real hypersurface in $\CP^2$ and let $p$ be a point in $S$.  Then there is an automorphism of $\CP^2$ moving $p$ to $0\in\C^2$ so that the transformed $S$ takes the form 
\begin{equation} \label{E:normal}
\Im z_2= |z_1|^2 + \beta \Re z_1^2 
+ O\left(\|(z_1,\Re z_2)\|^3\right)
\end{equation}
near $0$ with uniquely-determined $\beta\in[0,\infty)$.
\end{Prop}

\begin{proof}
See [Bar, Prop.\,5] and the following discussion.
\end{proof}

(For projective normalization of higher-order terms see [Ham].)

For a hypersurface $S$ of the form \eqref{E:normal} we have (by direct calculation) that
\begin{equation*}\label{E:belt-beta}
\belt_S(0) = \beta\, \frac{dz_1\w dz_2}{\overset{}{d\bar{z_1}\w d\bar{z_2}}};
\end{equation*}
moreover the order of contact between $S$ and $S_{\sf Heis}$ is 
$\begin{cases}
\ge3 &\text{ if }\beta= 0\\
2 &\text{ if }\beta\ne 0.
\end{cases}$
Thus the {\em projective-umbilic points} from the end of \S \ref{S:back} are precisely the points where $\belt_S$ vanishes.

\begin{Ex}\label{X:Lp}
The smooth portion $\{(z_1,z_2)\st |z_1|^p + |z_2|^p=1, z_1z_2\ne 0\}$ of the boundary of the unit $L^p$ ball in $\C^2$ is locally CR-equivalent to the sphere, using a branch of $\left(z_1^{2/p},z_2^{2/p}\right)$, but contains no projective-umbilic points when $p\ne2$ since  
\begin{equation*}
\belt_S = \frac{2-p}{p}\frac{\bar{z_1 z_2}}{z_1 z_2}
\,\frac{dz_1\w dz_2}{\overset{}{d\bar{z_1}\w d\bar{z_2}}}
\end{equation*}
in this case.
\end{Ex}

\section{Main result}\label{S:main}

\begin{Thm}\label{T:main}
If $S$ is the boundary of a  bounded strongly pseudoconvex complete circular domain in $\C^2$ then $S$ contains at least one circle of projective-umbilic points.
\end{Thm}

\begin{proof}
The complete circularity condition implies in particular that $S$ intersects each complex line $L$ through the origin in one circle $C_L$; since $S$ is strongly pseudoconvex,  $L$ cannot be completely tangent to $S$ along $C_L$, so in fact $L$ intersects $S$ transversely along $C_L$.

It follows that we may write
\begin{equation}\label{E:no-z1-ax}
S\cut \{z_2=0\}=\left\{\left(z_1,z_2\right)\st z_2\ne0, z_2\bar{z_2}=e^{\rho\left(z_1/z_2\
\right)}
\right\}
\end{equation}
where $\rho$ is a smooth $\R$-valued function on $\C$. Similarly we may write 
\begin{equation} \label{E:no-z2-ax}
S\cut \{z_1=0\}=\left\{\left(z_1,z_2\right)\st z_1\ne0, z_1\bar{z_1}=e^{\tilde\rho\left(z_2/z_1
\right)}
\right\}.
\end{equation}

Setting $\zeta=z_1/z_2$ we have from \eqref{E:no-z1-ax} that
\begin{equation*}
\belt_S =  b_S(\zeta)
\frac{\bar{z_2}^2}{z_2^2}
\,\frac{dz_1\w dz_2}{\overset{}{d\bar{z_1}\w d\bar{z_2}}}
\end{equation*}
where
$b_S(\zeta)=-\dfrac{\rho_{\zeta\zeta}-\rho_\zeta^2}{\rho_{\zeta\bar\zeta}}$ is $\C$-valued (but not holomorphic).
The strong pseudoconvexity of $S$ guarantees that the denominator $\rho_{\zeta\bar\zeta}$ is non-vanishing for $\zeta\in\C$.

Using \eqref{E:no-z2-ax} instead we have the alternate formula 
\begin{equation*}
\belt_S =  \tilde b_S(1/\zeta)\,
\,\frac{\bar{z_1}^2}{z_1^2}
\frac{dz_1\w dz_2}{\overset{}{d\bar{z_1}\w d\bar{z_2}}}.
\end{equation*}

Comparing the formulae we find that
\begin{align*}
b_S(\zeta) &= \tilde b_S(1/\zeta)\cdot \frac{\bar\zeta^2}{\zeta^2}
\overset{\zeta\text{ large}}\approx \tilde b_S(0) \cdot \frac{\bar\zeta^2}{\zeta^2}
\end{align*}
Assuming that $S$ is not projective-umbilic at points lying on the $z_1$-axis we have $\tilde b_S(0)\ne0$.

It follows that the logarithmic integral $\intl_{|\zeta|=M}\frac{db_S}{b_S}\in 2\pi i\Z$ must equal $-8\pi i$ for $M$ large.  From Stokes' theorem we now see that $b_S$ must have zeros in the disk ${|\zeta|<M}$.
\end{proof}

\section{Comments and examples.}\label{S:rmk}

\begin{Ex}\label{X:bulge}
Consider the hypersurface $S=\left\{\left(z_1,z_2\right)\st \left( |z_1|^2+|z_2|^2\right)^2+|z_1|^4+|z_2|^4=2\right\}$.
Computation reveals that 
\begin{equation*}
\belt_S=-\frac{3\bar{z_1^2z_2^2}}
{2\left(|z_1|^4+4|z_1z_2|^2+|z_2|^4\right)}
\,\frac{dz_1\w dz_2}{\overset{}{d\bar{z_1}\w d\bar{z_2}}},
\end{equation*}
so $\belt_S$ has double (conjugate) 
zeros along each axis.
In fact, $S$ has fourth order contact with $2|z_1|^2+|z_2|^2=2$ along the $z_1$-axis and with 
$|z_1|^2+2|z_2|^2=2$ along the $z_2$-axis.
\end{Ex}

\begin{Q}\label{Q:q} Suppose that $S\subset\CP^2$ is a not-necessarily-circular compact strongly pseudoconvex real hypersurface satisfying the {\em strong $\C$-convexity condition} $\left|\belt_S\right|<1$ ([APS, Def.\,2.5.10, [Bar], \S\S 5.2-3]).  Must $S$ have a projective-umbilic point? 
\end{Q}

\begin{Ex}
\label{X:torus} The answer to the above question is negative if the 
strong $\C$-convexity condition is dropped.  In fact, the example
\begin{equation*}
\left( \log|z_1| \right)^2 + \left( \log|z_2| \right)^2=\eps^2 
\end{equation*}
from [ESZ] of a compact strongly pseudoconvex hypersurface in $\C^2$ without CR-umbilic points also has no projective-umbilic points when $\eps$ is small.  (The latter claim follows from $\beta_S=-\frac{\bar{z_1z_2}}{z_1z_2}(1+\mathcal{O}(\eps))\,
\frac{dz_1\w dz_2}{\overset{}{d\bar{z_1}\w d\bar{z_2}}}$.)

{\em Note:} The strong $\C$-convexity condition appearing in Question \ref{Q:q} implies in particular that the domain bounded by $S$ is homeomorphic to the unit ball [APS, Thm.\,2.4.2].

\end{Ex}

\begin{Remark}
\label{R:top-geo}
The proof of Theorem \ref{T:main} given above is essentially topological.  In effect, it shows that any Beltrami differential
on $S^3$ that is invariant under rotations $R_\theta\st \left(z_1,z_2\right)\mapsto \left(e^{i\theta_1} z_1, e^{i\theta_2} z_2\right)$ must vanish along at least one circle.

It will not be possible to resolve Question \ref{Q:q} by a purely topological argument; in fact, on any smooth real  hypersurface in $\C^2$ we have the nowhere-vanishing Beltrami differential $\frac{dz_1\w dz_2}{\overset{}{d\bar{z_1}\w d\bar{z_2}}}$.

It is worth noting here that most  of the results
mentioned in \S \ref{S:back} require a proof with genuine geometry, not just topology (though topological arguments often suffice to prove weaker versions).
\end{Remark}

\section{Competing CR structures with the same maximal complex subspace}\label{S:Rossi}

Let $S\subset\C^2$ be a smooth connected real hypersurface with defining function $r$.  The maximal complex subbundle $HS\subset TS$ may be described as $\ker\left(  d^cr\rest_{TS}\right)$ where $d^c=\frac{\dee-\deebar}{2i}$. 

Suppose that we have an alternate CR structure on $S$ with the same maximal complex subbundle $HS$.

Let $\omega$ be a nowhere-vanishing 1-form on $S$ that is type $(1,0)$ on $HS$ with respect to the alternate CR structure.  Then $\omega$ may be uniquely decomposed as $\omega'+\omega''$ where $\omega'$ is type (1,0) on each $H_pS$ with respect to the original CR structure and $\omega''$ is type (0,1) on each $H_pS$ with respect to the original CR structure.  We have $|\omega''|<|\omega'|$ if the orientations on $HS$ match and $|\omega''|>|\omega'|$ if they do not match.

The 2-forms $d^cr\w \omega'$ and $d^cr\w \omega''$ may be extended to forms on a neighborhood of $S$ of types (2,0) and (0,2), respectively, with respect to the original CR structure.

Replacing $\omega$  by $\tilde\omega=\lam\omega$ (with $\lam$ nowhere vanishing) has the effect of multiplying the (2,0)- and (0,2)-forms above by $\lam$ along $S$ so that the ratio is unchanged along $S$.

If the orientations on $HS$ match then to avoid a vanishing denominator we should take the (2,0)-form to be the denominator.  By the same reasoning, in the other case we should take the (0,2)-form to be the denominator; this is the situation arising behind the scenes earlier in this paper, where the alternate CR structure is the one induced by projective duality considerations as in \S3 of [BaGr], leading to the ratio \eqref{E:belt} above.  

The alternate CR structure can also be defined by the orientation choice together with a dilation-invariant family of ellipses in each $H_pS$ which correspond to circles for the alternate structure. 
In the case of non-matching orientations 
the magnitude $\left| b(p) \frac{dz_1\w dz_2}{\overset{}{d\bar{z_1}\w d\bar{z_2}}} \right|\eqdef \left| b(p)\right|<1$ determines the major-to-minor axis ratio $\dfrac{1+|b(p)|}{1-|b(p)|}$ for the ellipses in $H_pS$, while  a vector $X\in H_p(S)\cut\{0\}$ points in the direction of the minor axes precisely when $b(p) \frac{dz_1\w dz_2}{\overset{}{d\bar{z_1}\w d\bar{z_2}}}(X,Y)>0$ for some (equivalently, for all) $Y\in T_pS\cut H_pS$ -- see [Bar, \S 5.3].  (Here $T_pS$ is the space of {\em real} tangent vectors to $S$.)    In the case of matching orientations we reach corresponding conclusions starting with $ b(p) \frac{\overset{}{d\bar{z_1}\w d\bar{z_2}}}{dz_1\w dz_2} $  (again with $\left| b(p)\right|<1$).

To provide a concrete illustration we consider a famous example of Rossi [Ros]  (see also [Bur])
of a family of competing CR structures on the unit sphere $S^3$ in $C^2$.

The standard CR structure on $S^3$ can be described by the condition that the CR functions on $S^3$ are those annihilated by the (complex) tangential vector field $\bar L \eqdef z_2 \frac{\dee}{\dee \bar z_1}-z_1 \frac{\dee}{\dee \bar z_2}$.  The CR functions for the alternate structure (which depends on a real parameter $t\in(-1,1)$) are those annihilated by $\bar L + t L$.
The two structures share the same $H_{(z_1,z_2)}S^3=\left\{\left(\gamma \bar z_2, - \gamma \bar z_2)\right)\st \gamma\in \C\right\}$, with matching orientation.  (Here we are viewing $H_{(z_1,z_2)}S^3$ as a  vector subspace of the real tangent space $T_{(z_1,z_2)}\C^2$, which we identify with $\C^2$.)

The function $f=z_1^2+z_2^2 + t\left(\bar z_1^2 + \bar z_2^2\right)$ is CR for the alternate structure, so the map
\begin{align*}
\mapdef{d_{(z_1,z_2)}f\st H_{(z_1,z_2)} S^3}{\C}{\left(\gamma \bar z_2,-\gamma \bar z_1\right)}
{\gamma\left( z_1\bar z_2 - \bar z_1z_2\right)+t\,\bar\gamma\left( \bar z_1 z_2 -  z_1\bar z_2\right)}
\end{align*}
is $\C$-linear for the alternate structure.
So the ellipses we want are given by \[\left| \gamma\left( z_1\bar z_2 - \bar z_1z_2\right)+\bar\gamma\,t\left( \bar z_1 z_2 -  z_1\bar z_2\right) \right|=C^{\sf constant}\]
with the minor axis corresponding to $\gamma\in i\R$ and the major-to-minor axis ratio equal to $\dfrac{1+t}{1-t}$.

(The argument above runs into trouble at points where $d_{(z_1,z_2)}f=0$ on $H_{(z_1,z_2)}S^3$; this happens in particular when one of the $z_j$ is a real multiple of the other.  However, the  conclusions above can still be shown to hold at such points by replacing $f$ by 
$g=z_1^2-z_2^2 - t\left(\bar z_1^2 - \bar z_2^2\right)$
 or by $h=z_1z_2 - t \bar z_1\bar z_2$.)

To compare this conclusions to the geometric discussion of Beltrami differentials above we see that
we must have $\dfrac{1+|b(p)|}{1-|b(p)|}=\dfrac{1+t}{1-t}$ hence $|b|\equiv t$ on $S^3$.  To check the direction of the minor axes we note that the vector
$\left\{\left(\gamma \bar z_2, - \gamma \bar z_1)\right)\st \gamma\in \C\right\} ]\in H_{(z_1,z_2)}S^3$ points in the direction of the minor axes when $\gamma\left( \bar z_1 z_2 -  z_1\bar z_2\right)\in\R$, in particular when $\gamma=i$; thus the vector field $X\eqdef \left(i \bar z_2,-i\bar z_1\right)$ describes the minor axes.

Taking $Y$ to be the rotational vector field 
\begin{align*}
\mapdef{ S^3}{TS^3}{\left(z_1,z_2\right)}{\left(iz_1,iz_2\right)}
\end{align*}
we may rewrite $X$ and $Y$ in operator form 
\begin{align*}
X &= i\left(\bar z_2\dfrac{\dee}{\dee z_1}- \bar z_1\dfrac{\dee}{\dee z_2}-z_2\dfrac{\dee}{\dee \bar z_1} + z_1\dfrac{\dee}{\dee \bar z_2} \right) \\
Y &= i\left(  z_1 \frac{\dee}{\dee z_1} + z_2 \frac{\dee}{\dee z_2}- \bar z_1 \frac{\dee}{\dee \bar z_1} - \bar z_2 \frac{\dee}{\dee \bar z_2} \right).
\end{align*}
Thus $dz_1\w dz_2 (X,Y) = -1 = d\bar z_1\w d\bar z_2$ and 
$ \frac{\overset{}{d\bar{z_1}\w d\bar{z_2}}}{dz_1\w dz_2}(X,Y)>0$.

Combining our conclusions, we see that the Beltrami differential for the Rossi example is 
$t \frac{\overset{}{d\bar{z_1}\w d\bar{z_2}}}{dz_1\w dz_2}$.  (Compare Remark \ref{R:top-geo} above.)

\end{document}